\documentclass{cimart}

\usepackage{cleveref}

\title{
   On the asymptotic behaviour of the graded-star-co\-di\-men\-sion sequence of upper triangular matrices
    }

\author{
    Diogo Diniz and Felipe Yasumura
    }

\authorinfo[
    D.~Diniz]{
    Federal University of Campina Grande, Brazil}{%
    diogo@mat.ufcg.edu.br
    }

\authorinfo[
    F.~Yasumura]{
    University of São Paulo, Brazil}{%
    fyyasumura@ime.usp.br
    }

\abstract{%
   We study the algebra of upper triangular matrices endowed with a group grading and a homogeneous involution over an infinite field. We compute the asymptotic behaviour of its (graded) star-codimension sequence. It turns out that the asymptotic growth of the sequence is independent of the grading and the involution under consideration, depending solely on the size of the matrix algebra. This independence of the group grading also applies to the graded codimension sequence of the associative algebra of upper triangular matrices.
    }



\VOLUME{33}
\YEAR{2025}
\ISSUE{1}
\NUMBER{3}
\DOI{https://doi.org/10.46298/cm.14050}
\begin{document}

\section{Introduction}
There is a recent interest in the investigation of polynomial identities with additional structure, in particular, the structure of a graded algebra and an involution. In the present paper, we investigate the asymptotic behaviour of the (graded)-star-codimension sequence of the associative algebra of upper triangular matrices $\mathrm{UT}_n$.

It is worth mentioning that the group gradings on $\mathrm{UT}_n$ are classified in \cite{VinKoVa2004,VaZa2007}. Their ordinary involutions were classified in \cite{VKS}, where the $\ast$-polynomial identities were computed for $n=2$ and $n=3$. Considering the mixture of the grading and the involution, their respective homogeneous involutions were classified in \cite{Mello}. In contrast, the graded-involutions were classified in \cite{FDGY} (see \cite{DRG} as well). The paper \cite{DRG} computes the $\ast$-graded polynomial identities of $\mathrm{UT}_n$, where the grading is the most fine that admits a graded-involution. On a similar direction, \cite{MY} computes the homogeneous $\ast$-polynomial identities of $\mathrm{UT}_n$, endowed with the unique fine grading of $\mathrm{UT}_n$ and an arbitrary homogeneous involution. In both papers, the authors compute the asymptotic behaviour of the graded-star-codimension sequence and the result is the same in each of the cases. Hence, the authors of the paper \cite{DRG} ask if it is possible to compute the asymptotic behaviour of the graded-star-codimension sequence of any graded-involution on $\mathrm{UT}_n$.

This paper answers the above question in a slightly more general context. We compute the asymptotic behaviour of $\mathrm{UT}_n$, endowed with a group grading and a homogeneous involution. We conclude that such behaviour is independent of the grading and of the involution in consideration. A similar phenomenon happens to the graded-codimension sequence, where the asymptotic behaviour of the graded-codimension sequence of $\mathrm{UT}_n$ is independent of the group grading (see \cite{KY}).

\section{Notation and basic results}
We denote by $\mathbb{F}$ an infinite field and all algebras in consideration will be over $\mathbb{F}$.
\subsection{Graded algebra} Let $G$ be any group. We use the multiplicative notation for $G$ and denote its neutral element by $1$. We say that an algebra $\mathcal{A}$ is $G$-graded if there exists a vector-space decomposition $\mathcal{A}=\bigoplus_{g\in G}\mathcal{A}_g$ such that $\mathcal{A}_g\mathcal{A}_h\subseteq\mathcal{A}_{gh}$, for all $g,h\in G$. The decomposition is called a \emph{$G$-grading} and we are going to denote it by $\Gamma$.

A homogeneous involution is defined as follows:
\begin{definition}
Let $\mathcal{A}=\bigoplus_{g\in G}\mathcal{A}_g$ be a $G$-graded algebra and let $\psi:G\to G$ be a map. An involution $\ast$ on $\mathcal{A}$ is a \emph{homogeneous involution} with respect to $\psi$ or a $\psi$-involution if $\mathcal{A}_g^\ast\subseteq\mathcal{A}_{\psi(g)}$, for all $g\in G$.
\end{definition}

\subsection{Free graded algebra with homogeneous involution}
We shall provide a construction of the free graded algebra endowed with a homogeneous involution (see, for instance, \cite{MY}). This is done using a particular case of the (relatively) free universal algebra in an adequate variety (see, for instance, \cite[Chapter 1]{Razmyslov} for a general discussion and \cite{BY,BY2} as well for a particular graded version). Let $G$ be any group and $X^G=\dot\bigcup_{g\in G}X^{(g)}$, where $X^{(g)}=\{x_1^{(g)},x_2^{(g)},\ldots\}$. Let $\ast:G\to G$ be an involution, that is, an anti-automorphism of order (at most) $2$. 
Let $\mathbb{F}\{X^G,\ast\}$ denote the absolutely free $G$-graded binary algebra endowed with an unary operation (also denoted by $\ast$).
We define the \emph{free $G$-graded associative algebra with a homogeneous involution} with respect to $\ast$, $\mathbb{F}\langle X^G,\ast\rangle$, 
as the quotient of $\mathbb{F}\{X^G,\ast\}$ by the following polynomials
\begin{align}
\begin{split}
&x_1^{(g_1)}(x_2^{(g_2)}x_3^{(g_3)})-(x_1^{(g_1)}x_2^{(g_2)})x_3^{(g_3)}\\%
&(x_1^{(g_1)}x_2^{(g_2)})^\ast-(x_2^{(g_2)})^\ast(x_1^{(g_1)})^\ast\\%
&((x^{(g)})^\ast)^\ast-x^{(g)}\\%
&\deg_G(x^{(g)})^\ast = (\deg_Gx^{(g)})^\ast.
\end{split}
\end{align}
The first polynomial defines the associativity relation while the second and third indicate that $\ast$ acts as an involution in the quotient algebra. The last connects the involution $\ast$ of the group and the unary operation $\ast$ of the algebra.

Given a $G$-graded algebra $(\mathcal{A},\Gamma)$ with a homogeneous involution $\ast$, we denote by $\mathrm{Id}(\mathcal{A},\Gamma,\ast)$ the set of all of its graded polynomial identities with involution. If $\Gamma=\Gamma_\mathrm{triv}$ is the trivial grading, then we recover the notion of $\ast$-polynomial identity. In this case, we denote by $\mathrm{Id}(\mathcal{A},\ast)$ the set of all $\ast$-polynomial identities of $(\mathcal{A},\ast)$.

\subsection{Codimension sequence} We recall some definitions concerning codimension sequence. Let $G$ be a group and consider a $G$-grading $\Gamma$ on an algebra $\mathcal{A}$. We assume that $\ast$ is a homogeneous involution on $(\mathcal{A},\Gamma)$. We denote
$$
P^{(\ast)}_m=\mathrm{Span}\left\{\left(x_{\sigma(1)}^{(g_1)}\right)^{\delta_1}\cdots\left(x_{\sigma(m)}^{(g_m)}\right)^{\delta_m} \ \middle| \ \sigma\in\mathcal{S}_m,g_1,\ldots,g_m\in G,\delta_\ell\in\{\emptyset,\ast\}\right\},
$$
where $\mathcal{S}_m$ is the symmetric group on the set of $m$ elements.

We define the $(G,\ast)$-codimension sequence by
$$
c_m(\mathcal{A},\Gamma,\ast)=\dim P_m^{(\ast)}/P_m^{(\ast)}\cap\mathrm{Id}(\mathcal{A},\Gamma,\ast).
$$
For the particular case where $\Gamma=\Gamma_\mathrm{triv}$ is the trivial grading, then we recover the notion of $\ast$-codimension sequence. In this case, we denote
$$
c_m(\mathcal{A},\ast)=c_m(\mathcal{A},\Gamma_\mathrm{triv},\ast).
$$
Two maps $f$, $g:\mathbb{N}\to\mathbb{N}$ are said to be \emph{asymptotically equal} if $\lim_{n\to\infty}\frac{f(n)}{g(n)}=1$. In this case, we denote $f\sim g$.

\subsection{Gradings and homogeneous involutions on \texorpdfstring{$\mathrm{UT}_n$}{\mathrm{UT}}}
The group gradings on $\mathrm{UT}_n$ are classified in two papers. Every group grading is isomorphic to an elementary grading \cite{VaZa2007} and isomorphism classes of elementary group gradings are known \cite[Theorem 2.3]{VinKoVa2004}.

The involutions of the first kind on $UT_n$ were described in \cite{VKS} over fields of characteristic not $2$. If $n$ is odd each involution on $UT_n$ is equivalent to a single involution $\tau$, which we will call the orthogonal involution. It is given by reflection along the secondary diagonal, that is, for each $i\leq j$, $e_{i,j}^\tau = e_{n-j+1,n-i+1}$. If $n = 2m$ is even, any involution on $UT_n$ is equivalent either to the orthogonal involution or to the symplectic involution. If we denote it by $s$ then it is given by $A^s = D A^\tau D^{-1}$, where $D=\mathrm{diag}(\underbrace{1,1,\ldots,1}_\text{$\frac{n}2$ times},\underbrace{-1,-1,\ldots,-1}_\text{$\frac{n}2$ times})$.

The homogeneous involutions on $\mathrm{UT}_n$ are classified in \cite{Mello} (and the graded involutions in \cite{FDGY}, see \cite{DRG} as well). It turns out that each homogeneous involution is equivalent to either the orthogonal or the symplectic involution (where the symplectic case can occur only if the size of the matrices is even). A homogeneous involution exists if and only if the grading satisfies certain conditions.

\subsection{Basis of the relatively free algebra of \texorpdfstring{$\mathrm{UT}_n$}{\mathrm{UT}}}
The following result is needed:
\begin{lemma}[{part of \cite[Theorem 5.2.1]{Drenskybook}}]\label{lem1}
Let $\mathbb{F}$ be an infinite field. Then, a basis of the relatively free algebra generated by $\mathrm{UT}_n$ consists of all polynomials of the form
$$
x_1^{p_1}\cdots x_r^{p_r}c_1\cdots c_s,
$$
where $p_1,\ldots,p_r\ge0$, $0\le s<n$ and each $c_i$ is a commutator of the kind $c_i=[x_{j_1},\ldots,x_{j_t}]$ and the indices satisfy $t\ge2$, $j_1>j_2$ and $j_2\le j_3\le\cdots\le j_t$.
\end{lemma}
From \Cref{lem1}, for a fixed integer $m\ge2(n-1)$, the set of multilinear ordinary polynomials in $m$ variables
$$
x_{i_1}\cdots x_{i_r}c_1\cdots c_{n-1},
$$
where $i_1<\cdots<i_r$, $c_i=[x_{j_1^{(i)}},x_{j_2^{(i)}},\ldots,x_{j_{s_i}^{(i)}}]$, $j_1^{(i)}>j_2^{(i)}<\cdots<j_{s_i}^{(i)}$, is linearly independent modulo $\mathrm{Id}(\mathrm{UT}_n)$. Let $q_m$ denote the number of such polynomials.

\begin{lemma}\label{lem2}
$\displaystyle q_m\sim c_m(\mathrm{UT}_n)\sim m^{n-1}n^{m-n+1}$.
\end{lemma}
\begin{proof}
One has that $c_m(\mathrm{UT}_n)=q_m+c_m(\mathrm{UT}_{n-1})$, from a basis of the relatively free algebra of $\mathrm{UT}_n$. It is known that (see \cite[Theorem 1.4]{BR})
$$
c_m(\mathrm{UT}_n)\sim m^{n-1}n^{m-n+1}.
$$
The result follows, since $\lim_{m\to\infty}\frac{c_m(\mathrm{UT}_{n-1})}{c_m(\mathrm{UT}_{n})}=0$.
\end{proof}

\section{Main result}
We let $\ast$ be an involution for $\mathrm{UT}_n$. It is known that we may assume that $\ast$ is either the orthogonal or the symplectic involution \cite[Proposition 2.5]{VKS}. 

We shall obtain a suitable lower bound for the $\ast$-codimension sequence of $\mathrm{UT}_n$.
\begin{lemma}
If $m\ge2(n-1)$, then the set of polynomials in $P_m^\ast$
\begin{equation}\label{pol}
x_{i_1}\cdots x_{i_r}c_1\cdots c_{n-1},
\end{equation}
where $i_1<\cdots<i_r$, $c_i=[x_{j_1^{(i)}}^{\delta_i},x_{j_2^{(i)}},\ldots,x_{j_{s_i}^{(i)}}]$, $j_1^{(i)}>j_2^{(i)}<\cdots<j_{s_i}^{(i)}$, $\delta_i\in\{\emptyset,\ast\}$ and $\delta_\ell=\emptyset$ for all $\ell>\lfloor\frac{n-1}2\rfloor$, is linearly independent modulo $\mathrm{Id}(\mathrm{UT}_n,\ast)$.
\end{lemma}
\begin{proof}
Denote the polynomials \eqref{pol} by $f_{I,J,\overline{\delta}}$, where $\overline{\delta}=(\delta_1,\ldots,\delta_{n-1})$, $I=\{i_1,\ldots,i_r\}$ and $J=\{(j_1^{(i)},\ldots,j_{s_i}^{(i)})\mid i=1,\ldots,n-1\}$. Let
$$
f=\sum_{I,J,\overline{\delta}}\lambda_{I,J,\overline{\delta}}f_{I,J,\overline{\delta}}\in\mathrm{Id}(\mathrm{UT}_n,\ast).
$$
Between the elements of the summand with a possibly nonzero $\lambda_{I,J,\overline{\delta}}$, choose a maximal set $I_0=\{i_1,\ldots,i_r\}$. We evaluate $x_{i_1}=\cdots=x_{i_r}=1$ (the identity matrix). Then, all $f_{I,J,\overline{\delta}}$ annihilate but the ones with $I_0\subseteq I$; i.e., $I_0=I$ since $I_0$ is maximal. We fix a $J_0=((j_1^{(i)},\ldots,j_{s_i}^{(i)})\mid i=1,\ldots,n-1)$. 

We fix a set $\overline{\delta}_0$ and we evaluate:
\begin{align*}
&x_{j_\ell^{(i)}}=\sum_{j=i+1}^ne_{jj},\quad\ell\ge2,\\%
&x_{j_1^{(i)}}^{\delta_i}=e_{i,i+1}.
\end{align*}
Without loss of generality, we may assume that $\overline{\delta}_0=(\emptyset,\ldots,\emptyset)$. If a polynomial $f_{I,J,\overline{\delta}'}$ gives a non-zero evaluation, then $I=I_0$ and $J$ coincides with $J_0$, except possibly in the first entry of each of their sequences. Therefore, we may write
$$
J=((k_1^{(i)},j_2^{(i)},\ldots,j_{s_i}^{(i)})\mid i=1,\ldots,n-1).
$$
To obtain a nonzero evaluation, we would need
$$
x_{k_1^{(i)}}^{\delta_i'}=e_{i,i+1},\quad \text{ for }i=1,2,\ldots,n-1,
$$
where $\overline{\delta}'=(\delta_1',\ldots,\delta_{n-1}')$. Thus, for each $i$, either $\delta_i'=\emptyset$ and $k_1^{(i)}=j_1^{(i)}$, or $\delta_i'=\ast$ and $k_1^{(i)}=j_1^{(n-i)}$. If $k_1^{(1)}=j_1^{(n-1)}$ and $\delta_1'=\ast$, then necessarily $k_1^{(n-1)}=j_1^{(1)}$ and $\delta_{n-1}'=\ast$. However, from definition, $\delta_\ell'=\emptyset$ for all $\ell>\lfloor\frac{n-1}2\rfloor$, a contradiction. Hence, $k_1^{(1)}=j_1^{(1)}$ and $\delta_1'=\emptyset$. An inductive argument shows that $J=J_0$ and $\overline{\delta}=\overline{\delta}_0$. Therefore, for such choice, there is a unique $f_{I,J,\overline{\delta}}$ that gives a non-zero evaluation. The result is proved.
\end{proof}

As a consequence, we obtain that $c_m(\mathrm{UT}_n,\ast)\ge2^{\lfloor\frac{n-1}2\rfloor}q_m$, for all $m\ge2(n-1)$. On the other hand, let $G$ be any group, $\Gamma$ a $G$-grading on $\mathrm{UT}_n$ and $\ast$ a homogeneous involution on $(\mathrm{UT}_n,\Gamma)$. We may assume, up to an automorphism, that every matrix unit is homogeneous with respect to $\Gamma$ and $\ast$ is either the symplectic or orthogonal involution. Now, let $\Delta$ be the unique fine grading on $\mathrm{UT}_n$. It means that $\Delta$ is an $F$-grading, where $F$ is the free group freely generated by $\{r_1,\ldots,r_{n-1}\}$ and we impose $\deg_\Delta e_{i,i+1}=r_i$. Then, $\ast$ is a homogeneous involution on $(\mathrm{UT}_n,\Delta)$ as well. The group homomorphism $\alpha:F\to G$, where $\alpha(r_i)=\deg_\Gamma e_{i,i+1}$ induces a $G$-grading ${}^\alpha\Delta$ on $\mathrm{UT}_n$, so that $(\mathrm{UT}_n,{}^\alpha\Delta,\ast)\cong(\mathrm{UT}_n,\Gamma,\ast)$. Hence, a similar argument of \cite[Lemma 3.1]{BahGR98} (see \cite{KY} as well) shows that
$$
c_m(\mathrm{UT}_n,\ast)\le c_m(\mathrm{UT}_n,\Gamma,\ast)\le c_m(\mathrm{UT}_n,\Delta,\ast).
$$
Furthermore, \cite[Theorem 20]{MY} proves that $c_m(\mathrm{UT}_n,\Delta,\ast)\sim\frac{2^{\lfloor\frac{n-1}2\rfloor}}{n^{n-1}}m^{n-1}n^m$. Hence, we proved
\begin{theorem}
Let $G$ be a group, $\mathbb{F}$ be an infinite field, $\Gamma$ a $G$-grading on $\mathrm{UT}_n$ and $\ast$ a homogeneous involution on $(\mathrm{UT}_n,\Gamma)$. Then
$$
c_m(\mathrm{UT}_n,\Gamma,\ast)\sim\frac{2^{\lfloor\frac{n-1}2\rfloor}}{n^{n-1}}m^{n-1}n^m.
$$
In particular, $\mathrm{exp}(\mathrm{UT}_n,\Gamma,\ast)=n$. \qed
\end{theorem}

In particular, we answer a question posed by D.~Diniz et al in \cite{DRG}:
\begin{corollary}
Let $G$ be a group, $\mathbb{F}$ be an infinite field, $\Gamma$ a $G$-grading on $\mathrm{UT}_n$ and $\ast$ a graded-involution on $(\mathrm{UT}_n,\Gamma)$. Then
$$
c_m(\mathrm{UT}_n,\Gamma,\ast)\sim\frac{2^{\lfloor\frac{n-1}2\rfloor}}{n^{n-1}}m^{n-1}n^m.
$$\qed
\end{corollary}
As a particular case, we obtain the asymptotic behaviour of the $\ast$-codimension sequence of $\mathrm{UT}_n$.
\begin{corollary}
Let $\mathbb{F}$ be an infinite field and $\ast$ a involution on $\mathrm{UT}_n$. Then
$$
c_m(\mathrm{UT}_n,\ast)\sim\frac{2^{\lfloor\frac{n-1}2\rfloor}}{n^{n-1}}m^{n-1}n^m.
$$\qed
\end{corollary}

\section*{Acknowledgments}
We express our gratitude to the anonymous referees for their valuable suggestions and corrections, which have improved the exposition of this paper. D.~Diniz was partially supported by FAPESQ grant No.~3099/2021 and CNPq grants No.~304328/2022-7 and No.~405779/2023-2. F.~Yasumura was supported by FAPESP, grants No.~2018/23690-6 and No.~2023/03922-8.

{\small\bibliography{cimart}}
\EditInfo{August 12, 2024}{October 29, 2024}{Ivan Kaygorodov}
\end{document}